\documentclass{article}
\usepackage{latexsym}
\usepackage{amsmath, amsthm, amscd, amsfonts,amssymb}

\newcommand{\A}{{\textbf{(A)}}}
\newcommand{\B}{{\textbf{(B)}}}
\newcommand{\C}{{\textbf{(C)}}}

\newcommand{\bS}{{\mathbb{S}}} 
\newcommand{\bZ}{{\mathbb{Z}}}

\newtheorem{thm}{Theorem}
\newtheorem{lem}[thm]{Lemma}
       
\begin{document}

\title{\bf{Landau's Theorem Revisited Again}}

\author{\Large K. B. Reid\footnote{Corresponding author} and M. Santana\\
Department of Mathematics\\
California State University San Marcos\\
San Marcos, CA 92096-0001\\
{\tt msantana@csusm.edu}\\
{\tt breid@csusm.edu}}

\date{}

\maketitle

\thispagestyle{empty}

\begin{center}
Dedicated to the memory of Ralph Stanton.
\end{center}

\begin{abstract}
We give a new proof of the sufficiency of Landau's conditions for a non-decreasing sequence of integers to be the score sequence of a tournament.  The proof involves jumping down a total order on sequences satisfying Landau's conditions and provides a $O(n^2)$ algorithm that can be used to construct a tournament whose score sequence is any in the total order.  We also compare this algorithm with two other algorithms that jump along this total order, one jumping down and one jumping up.
\end{abstract}

\section{Introduction}

Necessary and sufficient conditions for a non-decreasing sequence of $n$ integers to be the score sequence of some tournament of order $n$, first given in 1953 by Landau \cite{Landau}, is a fundamental result of tournament theory.  Most every study involving scores in tournaments involves this result, so it is one of the very important tools in the study of tournaments.  The necessity of the conditions is very easy to prove, and most treatments are essentially identical.  However, many different proofs of the sufficiency of the conditions have appeared in the literature.  A survey of many of those proofs up to 1996 was given by Reid \cite{R2}, and we summarize those developments next.  

Landau's original proof appeared in 1953.  Matrix considerations by Fulkerson \cite{Fulkerson} in 1960 led to a proof, discussed more recently by Brualdi and Ryser \cite{BrualdiRyser} in 1991.  Berge \cite{Berge} also gave a network flow proof that appeared in 1960.  Alway \cite{Alway} gave a proof in 1962.  Fulkerson \cite{Fulkerson2} gave a constructive proof via matrices in 1965.  Ryser's 1964 proof \cite{Ryser} is the proof that appears in Moon's 1968 monograph \cite{Moon}.  Also in 1968, Brauer, Gentry, and Shaw \cite{BGS} gave an inductive proof.  In 1978, Mahmoodian \cite{Mahmoodian} gave a proof that appears in the 1979 textbook \cite{BCL} by Behzad, Chartrand, and Lesniak-Foster.  A streamlined version of that proof appeared in 1981 by Thomassen \cite{Thomassen} and was adopted by Chartrand and Lesniak in subsequent revisions of their 1979 textbook, starting with their 1986 revision \cite{CL}.  A very nice proof using systems of distinct representatives appeared in 1979 by Bang and Sharp \cite{BangSharp}. Three years later in 1982, Achutan, Rao, and Ramachandra-Rao \cite{ARR} obtained a proof as a result of some slightly more general work.  In 1987, Bryant \cite{Bryant} gave a proof via a slightly different use of systems of distinct representatives.  Partially ordered sets were employed in a proof by Aigner \cite{Aigner} in 1984 and described by Li in 1986 (his version appeared in 1989 \cite{Li}).

Since 1996 several more proofs of sufficiency have appeared, two in a paper by Griggs and Reid \cite{GriggsReid}, one in 2009 by Brualdi and Kiernan \cite{BrualdiKiernan} using Rado's Theorem from matroid theory, and another inductive proof in 2009 by Holshouser and Reiter \cite{HolshouserReiter}.  In this paper we present a new proof of sufficiency in the vein of the two proofs by Griggs and Reid.  The main idea in these proofs is to utilize a total order on $n$-tuples of integers satisfying Landau's conditions and jump along this total order.  The first proof in \cite{GriggsReid} starts with a known score sequence and jumps down this total order to a given $n$-tuple.  The second proof in \cite{GriggsReid} starts with a given $n$-tuple and jumps up this total order to a known score sequence.  The proof in this paper starts with a given $n$-tuple and jumps down this total order to a known score sequence.  In all of these proofs, each $n$-tuple in the processes is a score sequence if and only if the $n$-tuple reached by a single jump is a score sequence.  We also explain why the maximum number of jumps possible in our algorithm is $\dfrac{n^2 - 1}{8}$, if $n$ is odd, or $\dfrac{n^2 - 2n}{8}$, if $n$ is even.  Our proof makes several uses of the necessity of Landau's conditions.  In Section \ref{sec:5}, we compare our algorithm with the two in \cite{GriggsReid} discussed above.

\section{Definitions and Basic Results}\label{definitions}

A \emph{tournament} $T = (V(T), A(T))$ is an orientation of a complete graph.  $V(T)$ (or simply $V$ if there is no confusion) denotes the vertex set of $T$, and $A(T)$ (of simply $A$ if there is no confusion) denotes the arc set of $T$.  If $|V(T)| = n$, $T$ is called an $n$\emph{-tournament}.  If $(x,y) \in A$, where $x,y \in V$, we say $x$ \emph{dominates} $y$ or $x$ \emph{beats} $y$ and denote this by $x \to y$.  So, $x \to y$ will be used both to indicate the arc $(x,y)$ and to indicate that $x$ dominates $y$.  If $x \in V(T)$, then the \emph{out-set of $x$}, denoted $O(x)$, is the set $\{y | y \in V, x \to y\}$, and the \emph{in-set of $x$}, denoted $I(x)$, is the set $\{z | z \in V, z \to x\}$.  The \emph{score of $x$} (or \emph{out-degree of $x$}) is the integer $|O(x)|$.  A tournament is \emph{strong} if there exists a path from any vertex to any other vertex.  The \emph{score sequence} of $T$ is the $n$-vector of the scores of the vertices arranged in non-decreasing order.  Additional material on tournaments can be found in the 2004 survey by Reid \cite{ReidSurvey}

We now state the aforementioned result by Landau.

\begin{thm}[Landau]
A sequence of integers $S = (s_1, s_2, \dots, s_n)$, where $s_1 \le s_2 \le \dots \le s_n$, is the score sequence of some $n$-tournament if and only if:

(1)  $\displaystyle\sum\limits_{i=1}^k s_i \ge \dbinom{k}{2}$, $1 \le k \le n$, and 

(2)  $\displaystyle\sum\limits_{i=1}^n s_i = \dbinom{n}{2}$.
\end{thm}

In the following we will refer to these conditions as condition \emph{(1)} and condition \emph{(2)}.

A result that is often given as a corollary to Landau's Theorem is the following result.

\begin{thm}
A non-decreasing sequence of $n$ integers $S = (s_1, s_2, \dots, s_n)$ is the score sequence of some strong $n$-tournament if and only if $\displaystyle\sum\limits_{i=1}^k s_i > \dbinom{k}{2}$, for all $1 \le k < n$ and $\displaystyle\sum\limits_{i=1}^n s_i = \dbinom{n}{2}$.
\end{thm}

The proof of this requires both the necessary and sufficient conditions of Landau's Theorem.  In this paper we use the following version of this theorem, a version that requires only the necessary conditions of Landau's Theorem.  Then we present some other lemmas to be used that are of interest in their own right.

\begin{lem}\label{thm:strongLandau}
A score sequence $S = (s_1, s_2, \dots, s_n)$ is the score sequence of some strong $n$-tournament if and only if $\displaystyle\sum\limits_{i=1}^k s_i > \dbinom{k}{2}$, for all $k, 1 \le k < n$.
\end{lem}

\begin{proof}
To prove this, we will instead prove the contrapositive.  That is, we will show that a score sequence $(s_1, s_2, \dots, s_n)$ is the score sequence of some $n$-tournament which is not strong if and only if there exists some $k$, $1 \le k < n$ such that $\displaystyle\sum\limits_{i=1}^k s_i = \dbinom{k}{2}$.

Suppose $S = (s_1, s_2, \dots, s_n)$ is the score sequence of some $n$-tournament, $T$, that is not strong.  Then $T$ contains at least two strong components.  Let $C_1$  be the strong component of $T$ such that all of the vertices in $T$ that are not in $C_1$ dominate all of the vertices in $C_1$.  Suppose also that $C_1$ contains exactly $k \ge 1$ vertices of $T$.  We first show that the vertices of $C_1$ have scores $s_1, s_2, \dots, s_k$.  Suppose on the contrary, there exists $i$ and $j, 1 \le i \le k < j$, such that $s_i$ is the score of a vertex $v_i$ not in $C_1$, and $s_j$ is the score of a vertex $v_j$ in $C_1$.  Since $v_j$ is in $C_1$ and $v_j$ beats no vertex outside of $C_1$, $s_j$ must be strictly less than $k$.  As $v_i$ is not in $C_1$, $v_i$ beats all of the vertices of $C_1$, so that $s_i \ge k$.  Yet, $i < j$ implies $s_i \le s_j$, so $k \le s_i \le s_j < k$, a contradiction.  Thus, the vertices in $C_1$ must have scores from $\{s_1, s_2, \dots, s_k\}$.  Notice that the vertices of $C_1$ induce a tournament of order $k$, which has score sequence $(s_1, s_2, \dots, s_k)$.  Since a tournament of order $k$ must have exactly $\dbinom{k}{2}$ arcs, the sum of these scores must be $\dbinom{k}{2}$.  That is, $\displaystyle\sum\limits_{i=1}^k s_i = \dbinom{k}{2}$.

Now, suppose $S = (s_1, s_2, \dots, s_n)$ is the score sequence of some $n$-tournament $T$, and there exists $k, 1 \le k < n$, such that $\displaystyle\sum\limits_{i=1}^k s_i = \dbinom{k}{2}$.  Also suppose that vertex $v_i$ in $T$ has score $s_i$, $1 \le i \le n$.  Let $V' = \{v_1, v_2, \dots, v_k\}$ and consider the tournament induced by the vertices of $V'$, $T[V']$.  Let $s_i'$ be the score of $v_i$ in $T[V']$, $1 \le i \le k$.  Notice that $s_i' \le s_i$, for all $i, 1 \le i \le k$.  As $T[V']$ is a tournament with exactly $\dbinom{k}{2}$ arcs, we need $\displaystyle\sum\limits_{i=1}^k s_i' = \dbinom{k}{2}$.  Yet, $\dbinom{k}{2} = \displaystyle\sum\limits_{i=1}^k s_i' \le \displaystyle\sum\limits_{i=1}^k s_i = \dbinom{k}{2}$.  Thus, $\displaystyle\sum\limits_{i=1}^k s_i' = \displaystyle\sum\limits_{i=1}^k s_i$, and so, $s_i' = s_i$ for all $i, 1 \le i \le k$.  Hence, each vertex of $T$ in $V'$ only dominates other vertices of $T$ also in $V'$.  So, each vertex outside of $V'$ must dominate all of the vertices within $V'$.  That is, $T$ is not strong.
\end{proof}

\begin{lem}\label{lem:contra}
Let $S = (s_1, s_2, \dots, s_n)$ be a non-decreasing sequence of integers satisfying \emph{(1)}.  For no $k, 1 \le k < n$ is $\displaystyle\sum\limits_{i=1}^k s_i = \dbinom{k}{2}$ and $s_k = s_{k+1}$.  
\end{lem}

\begin{proof}
Suppose on the contrary that such an integer exists, and let the smallest such be denoted $k$.  If $k = 1$, then $s_1 = \dbinom{1}{2} = 0$ and $s_2 = s_1 = 0$.  Clearly, $s_1 + s_2 = 0 < \dbinom{2}{2} = 1$, contrary to the fact that $S$ satisfies \emph{(1)}.

Now suppose $1 < k < n$.  Since $k$ is the smallest integer where $\displaystyle\sum\limits_{i=1}^k s_i = \dbinom{k}{2}$, we must have $\displaystyle\sum\limits_{i=1}^{k-1} s_i > \dbinom{k-1}{2}$.  So, $\dbinom{k-1}{2} + (k - 1) = \dbinom{k}{2} = \displaystyle\sum\limits_{i=1}^k s_i = \left(\displaystyle\sum\limits_{i=1}^{k-1} s_i\right) + s_k > \dbinom{k-1}{2} + s_k$.  This implies $s_k < k - 1$, and, as $s_{k+1} = s_k$, $s_{k+1} < k - 1$.  So, $\displaystyle\sum\limits_{i=1}^{k+1} s_i = \left(\displaystyle\sum\limits_{i=1}^k s_i\right) + s_k = \dbinom{k}{2} + s_k < \dbinom{k}{2} + k - 1 < \dbinom{k+1}{2}$, again a contradiction to \emph{(1)}.

In any case, we arrive at a contradiction.  Thus, no such $k$, exists.
\end{proof}

If $n$ is an odd integer, then any $n$-tournament in which all of the scores are as nearly equal as possible is called a \emph{regular $n$-tournament}.  If $n$ is an even integer, then any $n$-tournament in which all of the scores are as nearly equal as possible is called a \emph{nearly-regular $n$-tournament}.

It is easy to see that if $n$ is odd, then all of the scores in a regular $n$-tournament must be $\dfrac{n-1}{2}$, and if $n$ is even, then half of the scores in a nearly-regular $n$-tournament are $\dfrac{n}{2}$ and the other half of the scores are $\dfrac{n-2}{2}$.  Consequently, we see that the score sequence of any regular $n$-tournament is $\left(\dfrac{n-1}{2}, \dfrac{n-1}{2}, \dots, \dfrac{n-1}{2}\right)$, and the score sequence of any nearly-regular $n$-tournament is $\left(\dfrac{n-2}{2}, \dfrac{n-2}{2}, \dots, \dfrac{n-2}{2}, \dfrac{n}{2}, \dfrac{n}{2}, \dots, \dfrac{n}{2}\right)$.  We will refer to these as the \emph{regular score sequence} and \emph{nearly-regular score sequence}, respectively.  

In the following discussion we let $R_n = (r_1, r_2, \dots, r_n)$ denote the score sequence of a regular $n$-tournament if $n$ is odd, and we also let $R_n$ represent the score sequence of a nearly-regular $n$-tournament if $n$ is even.

When $n$ is odd, a regular $n$-tournament is easily constructed by taking its vertices to be the integers $0, 1, \dots, n - 1$, choosing a set $X$ of $\dfrac{n-1}{2}$ terms from the set $\{1, 2, \dots, n - 1\}$ so that the sum of any two elements in $X$ is different from 0 modulo $n$, and declaring that the out-set of each vertex $i$ is the set of the $\dfrac{n-1}{2}$ vertices $\{i + x | x\in X\}$, where addition is modulo $n$.  For example, take $X = \{1, 2, \dots, \dfrac{n-1}{2}\}$.  When $n$ is even, a nearly-regular $n$-tournament can be obtained by deleting a vertex of a regular $(n+1)$-tournament.

\begin{lem}\label{thm:nearlyreg}
Let $S = (a, a, \dots, a, a+1, a+1, \dots, a+1)$ be a sequence of $n$ integers that satisfies condition \emph{(2)}, where the multiplicy of $a$ is $m, 1 < m < n$.  Then $S$ must be the nearly-regular score sequence.
\end{lem}

\begin{proof}
Since $\dbinom{n}{2} = am + (n - m)(a+1)$, or $am + na + n - ma - m = \dfrac{n(n - 1)}{2}$, we see that $a = \dfrac{n-3}{2} + \dfrac{m}{n}$.  For $n$ odd, $\dfrac{n-3}{2} + \dfrac{m}{n}$ is not an integer.  So, $n$ is not odd, as $a$ is an integer.  Thus, $n$ must be even, and $a = \dfrac{n-2}{2} - \dfrac{1}{2} + \dfrac{m}{n}$.  Since $a$ is an integer we must have $\dfrac{m}{n} = \dfrac{1}{2}$, that is, $m = \dfrac{n}{2}$ and $a = \dfrac{n-2}{2}$.  So, $S = \left(\dfrac{n-2}{2}, \dfrac{n-2}{2}, \dots, \dfrac{n-2}{2}, \dfrac{n}{2}, \dfrac{n}{2}, \dots, \dfrac{n}{2}\right)$, the nearly-regular score sequence.
\end{proof}

An easy application of Lemma \ref{thm:strongLandau} shows that any regular $n$-tournament and any nearly-regular $n$-tournament is strong.

We now recall the well known 1-norm metric (or ``Manhattan'' metric) on $n$-tuples of real numbers.  Given two sequences of real numbers $S = (s_1, s_2, \dots, s_n)$ and $S' = (s_1', s_2', \dots, s_n')$, the distance between $S$ and $S'$, denoted by $d(S, S')$, is $\displaystyle\sum\limits_{i=1}^n |s_i - s_i'|$.  Another tool that we will use is the following.

\begin{lem}\label{thm:even}
Suppose that $S = (s_1, s_2, \dots, s_n)$ and $S' = (s_1', s_2', \dots, s_n')$ are two sequences of integers such that $\displaystyle\sum\limits_{i=1}^n s_i = \displaystyle\sum\limits_{i=1}^n s_i'$.  Then $d(S, S')$ is even.
\end{lem}

\begin{proof}
If $S = S'$, then clearly $d(S, S') = 0$.  So, assume $S \neq S'$.  Let $L = \{i | s_i' \le s_i\}$ and $G = \{i | s_i' > s_i\}$.  Clearly, $L$ and $G$ are disjoint and are subsets of $\{1,2, \dots, n\} = [n]$ such that $L \cup G = [n]$.  To show that each of $L$ and $G$ is non-empty, suppose first that $L = \emptyset$.  That is, $s_i' > s_i$ for all $i, 1 \le i \le n$.  Then $\displaystyle\sum\limits_{i=1}^n s_i' > \displaystyle\sum\limits_{i=1}^n s_i$, a contradiction.  On the other hand, suppose $G = \emptyset$, that is, $s_i' \le s_i$ for all $i, 1 \le i \le n$.  Since $S' \neq S$, there exists some $k, 1 \le k \le n$, such that $s_k' < s_k$.  Again, $\displaystyle\sum\limits_{i=1}^n s_i' < \displaystyle\sum\limits_{i=1}^n s_i$, a contradiction.  Thus, $L$ and $G$ partition the set $[n]$.

Now, $\displaystyle\sum\limits_{i \in L} |s_i - s_i'| =
\displaystyle\sum\limits_{i \in L} (s_i - s_i') =
\displaystyle\sum\limits_{i \in L} (s_i - s_i') + \displaystyle\sum\limits_{i \in G} (s_i - s_i') - \displaystyle\sum\limits_{i \in G} (s_i - s_i') = \displaystyle\sum\limits_{i = 1}^n (s_i - s_i') -  \displaystyle\sum\limits_{i \in G} (s_i - s_i') = \displaystyle\sum\limits_{i = 1}^n s_i - \displaystyle\sum\limits_{i =1}^n s_i' - \displaystyle\sum\limits_{i \in G} (s_i - s_i') = 
0 + \displaystyle\sum\limits_{i \in G} (s_i' - s_i) = 
\displaystyle\sum\limits_{i \in G} |s_i' - s_i| = \displaystyle\sum\limits_{i \in G} |s_i - s_i'|$.  Consequently, $d(S, S') = \displaystyle\sum\limits_{i =1}^n |s_i - s_i'| = \displaystyle\sum\limits_{i \in L} |s_i - s_i'| + \displaystyle\sum\limits_{i \in G} |s_i - s_i'| = 2\displaystyle\sum\limits_{i \in L} |s_i - s_i'|$.  Since $S$ and $S'$ are integer sequences, $\displaystyle\sum\limits_{i \in L} |s_i - s_i'|$ is an integer.  Thus, $d(S, S')$ is even.
\end{proof}

\section{Proof of Sufficiency}\label{proof}

Define $\bS_n = \{(s_1, s_2, \dots, s_n) | s_i \in \bZ, 0 \le s_1 \le s_2 \le \dots \le s_n, \displaystyle\sum\limits_{i=1}^k s_i \ge \dbinom{k}{2}, 1 \le k \le n, \text{and} \displaystyle\sum\limits_{i=1}^n s_i = \dbinom{n}{2}\}$.  So, $\bS_n$ is the set of non-decreasing integral $n$-tuples that satisfy Landau's conditions $\emph{(1)}$ and $\emph{(2)}$.  Define the order $\preceq$ on sequences in $\bS_n$ as follows:  for $A_n = (a_1, a_2, \dots, a_n)$ and $B_n = (b_1, b_2, \dots, b_n)$ in $\bS_n$, $A_n \preceq B_n$ if and only if either $A_n = B_n$, or $a_n < b_n$, or for some $i, 1 \le i < n$, $a_n = b_n, a_{n-1} = b_{n-1}, \dots, a_{i+1} = b_{i+1}, a_i < b_i$.  Then $\preceq$ is a total order on $\bS_n$ with maximum element $Tr_n = (0,1,2, \dots, n - 1)$, the score sequence of the transitive $n$-tournament (i.e., the $n$-tournament with no directed cycles), and with minimum element $R_n$, the score sequence of either a regular $n$-tournament, if $n$ is odd, or a nearly-regular score sequence, if $n$ is even.

We claim that the following algorithm will transform a given sequence in $\bS_n$ into the sequence $R_n$ via jumps down $(\bS_n, \preceq)$ in such a way that each sequence in the process is a score sequence if and only if the sequence reached by a single jump is a score sequence.  These jumps usually do not involve two sequences such that one covers the other in $(\bS_n, \preceq)$.  And, we will show that the first new sequence obtained is a strong score sequence.

\noindent\textbf{Algorithm}

\noindent\textbf{(i)}  Begin with $S_0 = (s_1, s_2, \dots, s_n) \neq R_n$ where $S_0 \in \bS_n$.  

\noindent\textbf{(ii)}  For $S_l = (s_1, s_2, \dots, s_n)$, $l \ge 0$, find indicies $p$ and $q$, $p < q$, where $s_1 = s_2 = \dots = s_p < s_{p+1}$ and $s_{q-1} < s_q = s_{q+1} = \dots = s_n$, and replace $s_p$ with $s_p + 1$ and $s_q$ with $s_q - 1$.  Note that $p$ may be 1 and $q$ may be $n$.

\noindent\textbf{(iii)}  Set this new $n$-tuple as $S_{l+1}$ and relabel the scores as $s_1, s_2, \dots, s_n$.

\noindent\textbf{(iv)}  If $S_{l+1} = R_n$, the regular or nearly-regular score sequence, then stop.  Otherwise, return to (ii).

If a non-decreasing integer sequence satisfies \emph{(1)} and \emph{(2)} and is neither the regular nor nearly-regular score sequence, then it is clear that such a $p$ and $q$, $p < q$, as in part (ii) exist.

Consider the following example where the appropriate $p$th and $q$th positions are underlined, and $A \succeq B$ means $B \preceq A$. 
\vspace{-12pt}
\begin{center}
$\begin{array}{l l l l l l l l}
& S_0 &= &(1,\underline{1},2,3,4,5,\underline{6},6) &
\succeq & S_1 &=& (\underline{1},2,2,3,4,5,5,\underline{6}) \\ 
\succeq &S_2 &= &(2,2,\underline{2},3,4,\underline{5},5,5) &
\succeq &S_3 &= &(2,\underline{2},3,3,4,4,\underline{5},5)\\
\succeq &S_4 &= &(\underline{2},3,3,3,4,4,4,\underline{5}) &
\succeq &S_5& = &(3,3,3,3,4,4,4,4).
\end{array}$
\end{center}
\vspace{-12pt}
Note that the jumps are not necessarily between sequences so that one covers the other in $(\bS_8, \preceq)$.  For example, $S_1 \succeq (2,2,2,2,4,5,5,6) \succeq S_2$, so $S_1$ does not cover $S_2$.

\begin{lem}\label{thm:12Landau}
In the algorithm above, for all $l \ge 0$, if $S_l \neq R_n$ is in $\bS_n$, then $S_{l+1}$ is in $\bS_n$.  Thus, every sequence obtained by the algorithm is in $\bS_n$.
\end{lem}

\begin{proof}
Let $S_{l+1} = (s_1', s_2', \dots, s_n')$.  Consider $\displaystyle\sum\limits_{i=1}^k s_i'$.  If $1 \le k < p$, then $\displaystyle\sum\limits_{i=1}^k s_i' = \displaystyle\sum\limits_{i=1}^k s_i \ge \dbinom{k}{2}$, since $S_l$ satisfies condition \emph{(1)}.  If $p \le k < q$, then $\displaystyle\sum\limits_{i=1}^k s_i' = \left(\displaystyle\sum\limits_{i=1, i \neq p}^k s_i'\right) + s_p' = \left(\displaystyle\sum\limits_{i=1, i \neq p}^k s_i\right) + s_p + 1 = \left(\displaystyle\sum\limits_{i=1}^k s_i\right) + 1 > \displaystyle\sum\limits_{i=1}^k s_i \ge \dbinom{k}{2}$, since $S_l$ satisfies condition \emph{(1)}.  Similarly, if $q \le k < n$, then $\displaystyle\sum\limits_{i=1}^k s_i' = \displaystyle\sum\limits_{i=1}^k s_i + 1 - 1 = \displaystyle\sum\limits_{i=1}^k s_i \ge \dbinom{k}{2}$, since $S_l$ satisfies condition \emph{(1)}.  Lastly, $\displaystyle\sum\limits_{i=1}^n s_i' = \displaystyle\sum\limits_{i=1}^n s_i + 1 - 1 = \displaystyle\sum\limits_{i=1}^n s_i = \dbinom{n}{2}$, since $S_l$ satisfies \emph{(2)}.  So, $S_{l+1}$ satisfies both condition \emph{(1)} and condition \emph{(2)}.

To show $s_1' \le s_2' \le \dots \le s_n'$, first consider the case where $p + 1 < q$.  By the definition of $S$, $s_1 = \dots = s_p < s_{p+1} \le \dots \le s_{q-1} < s_q  = \dots = s_n$.  So, $s_1 = \dots = s_{p-1} < s_p + 1 \le s_{p+1} \dots \le s_{q-1} \le s_q - 1 < s_{q+1} = \dots = s_n$.  This implies, $s_1' = \dots = s_{p-1}' < s_p' \le s_{p+1}' \dots \le s_{q-1}' \le s_q' < s_{q+1}' = \dots = s_n'$.  That is, $s_1' \le s_2' \le \dots \le s_n'$, as desired.  

Now, suppose $p +1 = q$.  That is, $s_1 = \dots = s_p < s_{p+1} = \dots = s_n$.  If $s_p + 1 < s_q$, then $s_1 = \dots = s_{p-1} < s_p+1 \le s_q - 1 < s_{q+1} = \dots = s_n$.  So, $s_1' = \dots = s_{p-1}' < s_p' \le s_q' < s_{q+1}' = \dots = s_n'$, and $s_1' \le s_2' \le \dots \le s_n'$, as desired.  

So, the only remaining possible case is if $p+1 = q$ and $s_p + 1 = s_q$.  This means that the terms of $S_l$ consist of exactly two distinct integers that differ by 1.  By Lemma \ref{thm:nearlyreg}, this implies that $S_l$ is the nearly-regular score sequence, a contradiction.  Thus, $S_{l+1}$ is non-decreasing.  Since $S_0$ satisfies \emph{(1)} and \emph{(2)} and is non-decreasing, the result follows by induction.
\end{proof}

Lemma \ref{thm:12Landau} shows that the algorithm preserves the conditions set out in Landau's Theorem.  Next, we show that the algorithm terminates with the regular or nearly-regular sequence $R_n$.

\begin{lem}\label{thm:alg1}
The algorithm produces a sequence of integral $n$-tuples beginning at $S_0$ and ending at $R_n$.
\end{lem}

\begin{proof}
Consider $S_l = (s_1, s_2, \dots, s_n) \neq R_n$ and $S_{l+1} = (s_1, \dots, s_{p-1}, s_p + 1, s_{p+1}, \dots, s_{q-1}, s_q - 1, s_{q+1}, \dots, s_n)$ for some $l \ge 0$ that appear in the algorithm.  

We first show $s_p < r_p$.  Suppose $n$ is odd and that, on the contrary, $s_p \ge r_p$.  Since $s_1 = s_p$ and $r_1 = r_p$, we deduce that $s_1 \ge r_1$.  This implies $s_i \ge r_i, 1 \le i \le n$.  Since $S_l \neq R_n$, there must exist some $k, 1 \le k \le n$ such that $s_i = r_i, 1 \le i < k$ and $s_j > r_j, k \le j \le n$.  So, $\displaystyle\sum\limits_{i=1}^n s_i > \displaystyle\sum\limits_{i=1}^n r_i = \dbinom{n}{2}$, a contradiction.  In the remainder of the proof that $s_p < r_p$, we suppose that $n$ is even.  If $s_p > r_p$, then $s_1 > r_1$ as $s_1 = s_p$ and $r_1 \le r_p$.  Since $s_1 \le s_2 \le \dots \le s_n$ and $r_i = r_1$ or $r_i = r_1 + 1$, we deduce that $s_i \ge r_i$ for all $i, 2 \le i \le n$.  Thus, $\displaystyle\sum\limits_{i=1}^n s_i > \displaystyle\sum\limits_{i=1}^n r_i = \dbinom{n}{2}$, a contradiction.  To show that $s_p = r_p$ leads to a contradiction in the case $n$ is even, we consider the following three possibilities for the index $p$: $1 \le p < \dfrac{n}{2}$, or $\dfrac{n}{2}+1 \le p < n$, or $p = \dfrac{n}{2}$.

Suppose $1 \le p < \dfrac{n}{2}$.  So, $s_1 = \dots s_p = r_p = r_1$.  That is, $s_i = r_i$ for all $i, 1 \le i \le p$.  Now, $s_p < s_{p+1}$ and $r_p = r_{p+1}$ since $p < \dfrac{n}{2}$.  Thus, $r_{p+1} < s_{p+1}$, and in fact, $r_n \le s_{p+1} \le s_{p+2} \le \dots \le s_n$.  So, for all $i$, $p+2 \le i \le n$, $r_i \le s_i$ since $r_i = r_p$ or $r_i = r_p + 1$.  Thus, $\displaystyle\sum\limits_{i=1}^n s_i > \displaystyle\sum\limits_{i=1}^n r_i = \dbinom{n}{2}$.

Next, suppose $\dfrac{n}{2} + 1 \le p \le n$.  So, $s_1 = \dots = s_p = r_p = \dots = r_n = r_1 + 1$ since $p \ge \dfrac{n}{2} + 1$.  This implies that $s_1 > r_1$ and $s_i \ge r_i$ for all $i, 2 \le i \le n$.  Thus, $\displaystyle\sum\limits_{i=1}^n s_i > \displaystyle\sum\limits_{i=1}^n r_i = \dbinom{n}{2}$.  

Lastly, suppose $p = \dfrac{n}{2}$.  This implies that $s_i = r_i$ for all $i, 1 \le i \le \dfrac{n}{2}$, and $s_i \ge r_i$ for all $i, \dfrac{n}{2}+1 \le i \le n$.  Since $S_l \neq R_n$, there exists $k, \dfrac{n}{2} + 1 \le k \le n$, such that $s_k > r_k$.  Consequently, $\displaystyle\sum\limits_{i=1}^n s_i > \displaystyle\sum\limits_{i=1}^n r_i = \dbinom{n}{2}$.

In any case we obtain $\displaystyle\sum\limits_{i=1}^n s_i > \dbinom{n}{2}$, a contradiction to the fact that $S_l$ satisfies \emph{(2)}, and consequently, $s_p < r_p$, as desired  A similar proof shows $r_q < s_q$.

Since $s_p < r_p$, then $s_p + 1 \le r_p$.  So, as $r_p - s_p = |r_p - s_p|$ and $r_p - (s_p + 1) = |r_p - (s_p + 1)|$, we deduce $|r_p - s_p| = r_p - s_p = r_p - (s_p + 1) + 1 = |r_p - (s_p +1)| +1$.  And since $r_q < s_q$, then $r_q \le s_q - 1$.  Using $s_q - r_q = |s_q - r_q|$ and $s_q - 1 - r_q = |s_q - 1 - r_q|$, we deduce $|s_q - r_q| = s_q - r_q = s_q - 1 - r_q + 1 = |(s_q - 1) - r_q| + 1$.

We now see that \\
$d(S_l, R_n) = \displaystyle\sum\limits_{i=1}^n |s_i - r_i| = \left(\displaystyle\sum\limits_{i=1; i \neq p,q}^{n} |s_i - r_i|\right) + |s_p - r_p| + |s_q - r_q| = \left(\displaystyle\sum\limits_{i=1; i \neq p,q}^{n} |s_i - r_i|\right) + |(s_p + 1) - r_p| + |(s_q - 1) - r_q| + 2 = d(S_{l+1}, R_n) + 2$.

Since $d(S_0, R_n)$ is even (by Lemma \ref{thm:even}) and the algorithm reduces the distance to $R$ by 2 at each step, we must eventually arrive at $R_n$.
\end{proof}

Of course, we desire each sequence obtained at each step of the algorithm to be a score sequence, so a tool to help establish that is given next.

\begin{lem}\label{thm:strong}
If, in the algorithm, $S_{l+1}$ is the score sequence of some strong $n$-tournament for some $l \ge 0$, then $S_l$ is the score sequence of some $n$-tournament (not necessarily strong).
\end{lem}

\begin{proof}
Let $S_{l+1} = (s_1, \dots, s_{p-1}, s_p + 1, s_{p+1}, \dots, s_{q-1}, s_q - 1, s_{q+1}, \dots, s_n) = (s_1', s_2', \dots, s_n')$ be the score sequence of some strong $n$-tournament $T'$ with vertices $v_1, v_2, \dots, v_n$, where $v_i$ has score $s_i'$ for all $i, 1 \le i \le n$.  Since $T'$ is strong, there exists a path from $v_p$ to $v_q$.  If we reverse the orientation of this path, we obtain a new $n$-tournament $T$, in which the only vertices whose scores are different are $v_p$ and $v_q$.  Specifically, $v_p$ will have score $s_p' - 1 = s_p$, and $v_q$ will have score $s_q' +1 = s_q$.  Since $s_{p-1} < s_p$ and $s_q < s_{q+1}$, $T$ will have score sequence $(s_1', \dots, s_{p-1}', s_p' - 1, s_{p+1}', \dots, s_{q-1}', s_q' + 1, s_{q+1}', \dots, s_n') = (s_1, s_2, \dots, s_n) = S_l$. 
\end{proof}

A non-strong score sequence that might result in Lemma \ref{thm:strong} occurs only in a special case as described next.

\begin{lem}\label{thm:notstrongend}
If, in the algorithm, $S_{l+1}$ is the score sequence of some strong $n$-tournament and $S_l$ is the score sequence of some $n$-tournament which is not strong, then $l = 0$ and $S_l = S_0$.
\end{lem}

\begin{proof}
Suppose, for some $l \ge 0$, that $S_{l+1}$ is the score sequence of some strong $n$-tournament.  By Lemma \ref{thm:strong}, $S_l$ is the score sequence of some $n$-tournament $T$.  We claim that if $T$ not strong, then $S_l =S_0$.  Aiming for a contradiction, suppose $S_l \neq S_0$.  Thus, $l$ must be at least 1, which implies $S_{l-1}$ exists.    

Recall that by the algorithm, there exists $p$ and $q, 1 \le p < q \le n$, such that $s_1 = s_2 = \dots = s_p < s_{p+1}$ and $s_{q-1} < s_q = \dots = s_n$.  So, $S_{l+1} = (s_1, s_2, \dots, s_{p-1}, s_p +1, s_{p+1}, \dots, s_{q-1}, s_q - 1, s_{q+1}, \dots, s_n)$.  Since $S_{l-1}$ exists, by the algorithm there must also exist $x$ and $y, 1 \le x < y \le n$, such that $S_{l-1} = (s_1, \dots, s_{x-1}, s_x - 1, s_{x+1}, \dots, s_{y-1}, s_y +1, s_{y+1}, \dots, s_n)$, where $s_1 = \dots = s_{x-1} = s_x - 1 < s_{x+1}$ and $s_{y-1} < s_y + 1 = s_{y+1} = \dots = s_n$.

To help clarify the following arguments, we may view $S_l$ either as $(s_1, \dots, s_{p-1}, s_p, s_{p+1}, \dots, s_{q-1}, s_q, s_{q+1}, \dots, s_n)$ in terms of $p$ and $q$ where $s_1 = \dots = s_p < s_{p+1}$ and $s_{q-1} < s_q = \dots = s_n$, or in terms of $x$ and $y$ where $s_1 = \dots = s_x < s_{x+1}$ and $s_{y-1} < s_y = \dots = s_n$ so that $S_l = (s_1, \dots, s_{x-1}, s_x, s_{x+1}, \dots, s_{y-1}, s_y, s_{y+1}, \dots, s_n)$.  

Now, the assumption that $S_l$ is not strong implies, by Lemma \ref{thm:strongLandau}, that there is a least positive integer $k < n$ such that $\displaystyle\sum\limits_{i=1}^k s_i = \dbinom{k}{2}$.  We now consider the location of $k$ in $S_l$ and aim to show a contradiction for all $k, 1 \le k < n$.

If $x \le k < y$, notice that the sum of the first $k$ terms in $S_{l-1}$ is exactly one less than $\displaystyle\sum\limits_{i=1}^k s_i = \dbinom{k}{2}$.  Yet, $S_{l-1}$ was produced by the algorithm, and by Lemma \ref{thm:12Landau}, must satisfy condition \emph{(1)}, a contradiction.  So, we arrive at a contradiction for all $k, x \le k < y$ (we refer to this statement as \A).

If $x \neq 1$, the two versions of $S_l$ clearly show $x = p+1$, and if $y \neq n$, the two versions of $S_l$ show that $y = q - 1$.  Thus, we consider are four cases: $x = 1$ or $p+1$ and $y = q - 1$ or $n$.

Suppose $x = p+1$ (regardless of the value of $y$).  If $1 \le k \le p - 1$, notice that the sum of the first $k$ terms in $S_{l+1}$ is equal to $\displaystyle\sum\limits_{i=1}^k s_i = \dbinom{k}{2}$.  This implies by Lemma \ref{thm:strongLandau}, that $S_{l+1}$ is not strong, a contradiction.  If $k = p = x - 1$, then notice that the sum of the first $x - 1$ terms in $S_{l-1}$ is equal to $\displaystyle\sum\limits_{i=1}^{x-1} s_i = \dbinom{x-1}{2}$.  By the definition of $x$, $s_{x-1} = s_x - 1$, which is the $x$th term in $S_{l-1}$.  Yet, by Lemma \ref{lem:contra}, this contradicts the assumption that $S_{l-1}$ satisfies condition \emph{(1)}.  Thus, if $x = p+1$, we arrive at a contradiction for all $k, 1 \le k \le p = x - 1$ (we refer to this statement as \B).

Now, suppose $y = q - 1$ (regardless of the value of $x$).  If $q \le k < n$, notice that the sum of the first $k$ terms in $S_{l+1}$ is equal to $\displaystyle\sum\limits_{i=1}^k s_i = \dbinom{k}{2}$.  This implies by Lemma \ref{thm:strongLandau}, that $S_{l+1}$ is not strong, a contradiction.  If $k = q - 1 = y$, then notice that the sum of the first $q - 1$ terms in $S_{l-1}$ is equal to $\displaystyle\sum\limits_{i=1}^{y} s_i = \dbinom{y}{2}$.  By the definition of $y$, $s_y+1 = s_{y+1}$, which is the $(y+1)$th term in $S_{l-1}$.  Yet, by Lemma \ref{lem:contra}, this contradicts the assumption that $S_{l-1}$ satisfies condition \emph{(1)}.  Thus, if $y = q - 1$, we arrive at a contradiction for all $k, y = q - 1 \le k < n$ (we refer to this statement as \C).

Finally, we explicitly treat the four cases mentioned above.  If $x = 1$ and $y = n$, then as we have shown above $x \le k < y$ produces a contradiction, so in this case a contradiction must hold for all $k, 1 \le k < n$.  Next, suppose $x = 1$ and $y = q - 1$.  By \A and \C, we arrive at a contradiction for all $k, 1 \le k < n$.  Next, suppose $x = p +1$ and $y = n$.  By \A and \B, we arrive at a contradiction for all $k, 1 \le k < n$. Lastly, suppose $x = p+1$ and $y = q - 1$.  By \A, \B, and \C, we arrive at a contradiction for all $k, 1 \le k < n$.

This exhausts all possible cases.  So, if $S_l$ is not strong, then $S_{l-1}$ does not exist.  Thus, $S_l$ must be $S_0$.
\end{proof}

Sufficiency is now a formality.

\begin{thm}
If $S = (s_1, s_2, \dots, s_n)$ is in $\bS_n$, then $S$ is the score sequence of some $n$-tournament.
\end{thm}

\begin{proof}
The theorem is clearly true if $S = R_n$.  So, suppose $S \neq R_n$.  By Lemma \ref{thm:alg1} the algorithm produces a sequence of $n$-tuples $S = S_0, S_1, \dots, S_M$, terminating in $S_M = R_n$ for some integer $M \ge 1$ (Actually $M = \frac{1}{2}d(R_n,S)$ by the last sentence of the proof of Lemma \ref{thm:alg1}).  By Lemma \ref{thm:12Landau}, $S_l$ satisfies both \emph{(1)} and \emph{(2)} with $s_1 \le s_2 \le \dots \le s_n$ for all $l, 0 \le l \le M$.  We now show, by induction on $j$, that $S_{M-j}$ is the score sequence of some strong $n$-tournament, for all $j, 0 \le j < M$.  If $j = 0$, then $S_{M-0}= S_M = R_n$, the regular or nearly-regular score sequence, which is strong.  Now, suppose $S_{M - j}$ is the score sequence of some strong $n$-tournament for some $j, 0 \le j < M - 1$.  Since $S_{M-j}$ is the score sequence of some strong $n$-tournament, by Lemma \ref{thm:strong}, $S_{M-j-1}$ is the score sequence of some $n$-tournament $T$.  If $T$ is not strong, then Lemma \ref{thm:notstrongend} implies that $S_{M-j-1} = S_0$.  That is, $M - j - 1 = 0$, a contradiction.  Thus, $T$ must be strong.

So, by induction, $S_l$ is the score sequence of some strong $n$-tournament, for all $l, 1 \le l \le M$.  In particular, $S_1$ is the score sequence of some strong $n$-tournament.  By Lemma \ref{thm:strong}, $S_0 = S$ is the score sequence of some $n$-tournament, as desired.
\end{proof}

Upon careful examination, the description and verification of the algorithm given above is much less cumbersome and involved than a description and verification of a possible ``inverse algorithm'' that starts with $R_n$ and ends with $S_0$ and that exactly reverses the steps of the algorithm given above.

\section{Complexity}

In the previous section we proved that if the integral sequence $S_0$ satisfies Landau's conditions, then $S_0$ is a score sequence.  In fact, the algorithm gives a constructive method by which any regular (or nearly-regular) tournament can be transformed into a tournament with such a score sequence $S_0$.  Indeed, if $T_{l+1}$ is any strong $n$-tournament with score sequence $S_{l+1}$, then, as there exists a path from any vertex to any other vertex in $T_{l+1}$, the reversal of a path in $T_{l+1}$ from a vertex of score $s_p+1$ to a vertex of score $s_q - 1$ results in a tournament $T_l$ with score sequence $S_l$.  So, the number of steps in the algorithm (i.e., the number of jumps down $(\bS_n, \preceq)$) is a measure of the complexity of the construction of an $n$-tournament from a regular (or nearly-regular) tournament to a tournament with score sequence $S_0$.  We claim that there is no case that requires more jumps than when $S_0$ is taken to be $Tr_n = (0,1,2, \dots, n - 1)$, the score sequence of the transitive $n$-tournament.  For any sequence $S$ satisfying Landau's conditions, the last sentence in the proof of Lemma \ref{thm:alg1} shows that the number of jumps in the algorithm is $\frac{1}{2}d(R_n,S)$.  So, the next result will confirm our claim.

\begin{thm}\label{thm:distance}
Let $S = (s_1, s_2, \dots, s_n)$ be the score sequence of some $n$-tournament and let $Tr_n = (0,1,2, \dots, n - 1)$.  Then $d(R_n, S) \le d(R_n,Tr_n)$.  
\end{thm}

\begin{proof}
The result is trivial for $S = R_n$ and $S = Tr_n$.  So, assume $S \neq R_n$ and $S \neq Tr_n$.  We break the proof into two cases where $R_n$ is the regular or nearly-regular score sequence.  First suppose that $n$ is odd so that $R_n = \left(\dfrac{n-1}{2}, \dfrac{n-1}{2}, \dots, \dfrac{n-1}{2}\right)$.  Recall in the proof of Lemma \ref{thm:even} that $d(R_n,S) = 2\displaystyle\sum\limits_{i \in L} |r_i - s_i| = 2\displaystyle\sum\limits_{i \in L} (r_i - s_i)$, where $L = \{i | s_i \le r_i\}$.  Since $R_n$ is the regular score sequence and $s_1 \le s_2 \le \dots \le s_n$, $d(R_n,S) = 2\displaystyle\sum\limits_{i =1}^k \left(\dfrac{n-1}{2} - s_i\right)$ for some $k$, $1 \le k < n$, where $s_i \le r_i$ for all $i \le k$, and $s_i > r_i$ for all $i$, $k < i \le n$.

It is easy to see that $d(R_n,Tr_n) = 2\displaystyle\sum\limits_{i=1}^{\frac{n+1}{2}} |\dfrac{n-1}{2} - (i - 1)| = 2\displaystyle\sum\limits_{i=1}^{\frac{n-1}{2}} i = 2\dbinom{\frac{n+1}{2}}{2} = \dfrac{n^2 - 1}{4}$.

We now show that for the integer $k$ above $\left(\dfrac{n-1}{2}\right)k - \dbinom{k}{2} \le \dfrac{n^2 - 1}{8}$.  Suppose on the contrary, $\left(\dfrac{n-1}{2}\right)k - \dbinom{k}{2} > \dfrac{n^2 - 1}{8}$.  This implies that $4k(n - 1) - 4(k^2 - k) > n^2 - 1$, or $4kn - 4k - 4k^2 + 4k - n^2 + 1 > 0$, or $k^2 - kn + \dfrac{n^2 - 1}{4} < 0$, or $\left(k - \dfrac{n+1}{2}\right)\left(k - \dfrac{n-1}{2}\right) < 0$.  So, either $k > \dfrac{n+1}{2}$ and $k < \dfrac{n-1}{2}$ or $k < \dfrac{n+1}{2}$ and $k > \dfrac{n-1}{2}$.  Since $k$ is an integer and $n$ is odd, both situations produce contradictions.  Thus, $\left(\dfrac{n-1}{2}\right)k - \dbinom{k}{2} \le \dfrac{n^2 - 1}{8}$, as desired.

This implies $d(R_n, S) = 2\displaystyle\sum\limits_{i=1}^k \left(\dfrac{n-1}{2} - s_i\right) = 2\left(\dfrac{k(n-1)}{2} - \displaystyle\sum\limits_{i=1}^k s_i\right) \le 2\left(\dfrac{k(n-1)}{2} - \dbinom{k}{2}\right)$.  The last inequality follows from the fact that $S$ is a score sequence and so satisfies \emph{(1)}.  By the above argument, we deduce $d(R_n,S) \le \dfrac{n^2 - 1}{4} = d(R_n,Tr_n)$, as desired.

Now suppose $n$ is even so that $R_n$ is the nearly-regular score sequence $\left(\dfrac{n-2}{2}, \dfrac{n-2}{2}, \dots, \dfrac{n-2}{2}, \dfrac{n}{2}, \dfrac{n}{2}, \dots, \dfrac{n}{2}\right)$.  It is easy to see that $d(R_n,Tr_n) = 2\displaystyle\sum\limits_{i=1}^{\frac{n}{2}+1} |r_i - (i - 1)| = 2\displaystyle\sum\limits_{i=1}^{\frac{n}{2}-1} i = 2\dbinom{\frac{n}{2}}{2} = \dfrac{n^2 - 2n}{4}$.

Let $k+1$ be the smallest positive integer such that $s_{k+1} > r_{k+1}$.  Either $
\dfrac{n}{2} +1 \le k +1 \le n$ or $1 < k+1 \le \dfrac{n}{2}$.  We treat both cases.  

Suppose $\dfrac{n}{2}+1 \le k+1 \le n$, so that $s_{k+1} > \dfrac{n}{2}$.  Since $s_1 \le s_2 \le \dots \le s_n$, $s_j > r_j = \dfrac{n}{2}$ for all $j$, $k+1 \le j \le n$.  Since $k+1$ is the first index of $S$ such that $s_{k+1} > r_{k+1}$, it follows that $d(R_n, S) = 2\displaystyle\sum\limits_{i \in L} |r_i - s_i| = 2\displaystyle\sum\limits_{i=1}^k (r_i - s_i)$.  Let $\alpha$ be a non-negative integer such that $\alpha + \dfrac{n}{2} = k$.  So, $d(R_n,S) = 2\displaystyle\sum\limits_{i=1}^k (r_i - s_i) = 2\left(\dfrac{n}{2}\left(\dfrac{n}{2} - 1\right) + \alpha\left(\dfrac{n}{2}\right) - \displaystyle\sum\limits_{i=1}^k s_i\right)$.  To complete this case, we now show that $\dfrac{n}{2}\left(\dfrac{n}{2} - 1\right) + \alpha\left(\dfrac{n}{2}\right) - \dbinom{\frac{n}{2}+\alpha}{2} \le \dfrac{n^2 - 2n}{8}$.  Aiming for a contradiction, suppose that $\dfrac{n}{2}\left(\dfrac{n}{2} - 1\right) + \alpha\left(\dfrac{n}{2}\right) - \dbinom{\frac{n}{2}+\alpha}{2} > \dfrac{n^2 - 2n}{8}$.  This implies that $\dfrac{n^2}{4} - \dfrac{n}{2} + \dfrac{n\alpha}{2} - \dfrac{(\frac{n}{2} + \alpha)(\frac{n}{2} + \alpha - 1)}{2} > \dfrac{n^2 - 2n}{8}$, or $2n^2 - 4n + 4n\alpha - 4\left(\dfrac{n}{2} + \alpha\right)\left(\dfrac{n}{2} + \alpha - 1\right) > n^2 - 2n$, or $n^2 - 2n +4n\alpha - n^2 - 2n\alpha + 2n - 2n\alpha - 4\alpha^2 +4\alpha > 0$, or $- 4\alpha^2 + 4\alpha > 0$, or $\alpha(\alpha - 1) < 0$.  So, either $\alpha < 0$ and $\alpha > 1$ or $\alpha > 0$ and $\alpha < 1$.  Since $\alpha$ is a non-negative integer, both situations produce contradictions.  So, $\dfrac{n}{2}\left(\dfrac{n}{2} - 1\right) + \alpha\left(\dfrac{n}{2}\right) - \dbinom{\frac{n}{2}+\alpha}{2} \le \dfrac{n^2 - 2n}{8}$, as desired.  Thus, $d(R_n,S) = 2\left(\dfrac{n}{2}\left(\dfrac{n}{2} - 1\right) + \alpha\left(\dfrac{n}{2}\right) - \displaystyle\sum\limits_{i=1}^k s_i\right) \le 2\left(\dfrac{n}{2}\left(\dfrac{n}{2} - 1\right) + \alpha\left(\dfrac{n}{2}\right) - \dbinom{\frac{n}{2} +\alpha}{2}\right)$.  The last inequality follows from the fact that $S$ is a score sequence and so, satisfies \emph{(1)}.  By the argument above, we deduce that $d(R_n,S) \le \dfrac{n^2-2n}{4} = d(R_n,Tr_n)$, in the case $\dfrac{n}{2}+1 \le k+1 \le n$.

We now treat the other case for the value of $k+1$.  Suppose $1 < k+1 \le \dfrac{n}{2}$, that is $s_{k+1} > \dfrac{n}{2} - 1$.  Note that $s_i > r_i$ for all $i$, $k+1 \le i \le \dfrac{n}{2}$.  However, it is possible that there exists some $j$, $\dfrac{n}{2}+1 \le j \le n$, such that $s_j \le r_j = \dfrac{n}{2}$.  That is, $j \in L = \{i | s_i \le r_i\}$.  Since $\dfrac{n}{2} < s_{k+1} \le s_j \le \dfrac{n}{2}$, the only way this can occur is if $s_j = \dfrac{n}{2}$.  So, for all such $j$, $\dfrac{n}{2}+1 \le j \le n$ where $s_j = \dfrac{n}{2}$, we see that $|r_j - s_j| = 0$.  Thus, when we consider $\displaystyle\sum\limits_{i \in L} |r_i - s_i|$, we can ignore all indices in $L$ which are greater than $k$.  So, because $s_1 \le s_2 \le \dots \le s_n$ and $s_k \le r_k = \dfrac{n}{2} - 1$ implies that $s_i \le r_i$ for all $i$, $1 \le i \le k$, we deduce $\displaystyle\sum\limits_{i \in L} |r_i - s_i| = \displaystyle\sum\limits_{i = 1}^k |r_i - s_i|$.  Thus, $d(R_n,S) = 2\displaystyle\sum\limits_{i=1}^k (r_i - s_i) = 2\left(\left(\dfrac{n}{2} - 1\right)k - \displaystyle\sum\limits_{i=1}^k s_i\right)$.  To complete this case we show that $\left(\dfrac{n}{2} - 1\right)k - \dbinom{k}{2} \le \dfrac{n^2 - 2n}{8}$.  Suppose, to the contrary, that $\left(\dfrac{n}{2} - 1\right)k - \dbinom{k}{2} > \dfrac{n^2 - 2n}{8}$.  This implies that $\dfrac{n}{2}k - k - \dfrac{k^2 - k}{2} > \dfrac{n^2-2n}{8}$, or $4nk - 8k - 4k^2 + 4k > n^2 - 2n$, or $- 4k^2 - 4k - 4nk +n^2 - 2n > 0$, or $k^2 - (n - 1)k + \dfrac{n^2 - 2n}{4} < 0$, or $\left(k - \dfrac{n}{2}\right)\left(k - \left(\dfrac{n}{2} - 1\right)\right) < 0$.  So, either $k > \dfrac{n}{2}$ and $k < \dfrac{n}{2} - 1$, or $k < \dfrac{n}{2}$ and $k > \dfrac{n}{2} - 1$.  Since $k$ is an integer and $n$ is even, both situations produce contradictions.  Consequently, $\left(\dfrac{n}{2} - 1\right)k - \dbinom{k}{2} \le \dfrac{n^2 - 2n}{8}$, as desired.  Thus, $d(R_n,S) = 2\left(\left(\dfrac{n}{2} - 1\right)k - \displaystyle\sum\limits_{i=1}^k s_i\right) \le 2\left(\left(\dfrac{n}{2} - 1\right)k - \dbinom{k}{2}\right)$.  The last inequality follows from the fact that $S$ is a score sequence and so, satisfies \emph{(1)}.  By the argument above, we deduce that $d(R_n,S) \le \dfrac{n^2-2n}{4} = d(R_n,Tr_n)$.

In any case, we see that $d(R_n,S) \le d(R_n, Tr_n)$, as desired.
\end{proof}

In summary, as explained prior to Theorem \ref{thm:distance}, the maximum number of down jumps that the algorithm will produce is $\dfrac{d(R_n, Tr_n)}{2}$ which is $\dfrac{n^2 - 1}{8}$ if $n$ is odd, or $\dfrac{n^2 - 2n}{8}$ if $n$ is even.  That is, it is a $O(n^2)$ algorithm.

\section{The total order $(\bS_n, \preceq)$ and two other proofs}\label{sec:5}

Let $S \in \bS_n, S \neq Tr_n$.  The first proof in \cite{GriggsReid} to show that $S$ is a score sequence starts with $Tr_n$ and jumps down the totally ordered set $(\bS_n, \preceq)$ to $S$ so that a sequence is a score sequence if and only if the sequence reached by a jump is a score sequence.  If $U \neq S$ is the current sequence in this process (starting with $U = Tr_n$), then the next jump is determined via three indices.  Let $\alpha$ denote the smallest index such that $u_{\alpha} < s_{\alpha}$, let $\beta$ denote the largest index such that $u_{\beta} = u_{\alpha}$, and let $\gamma$ denote the smallest index such that $u_{\gamma} > s_{\gamma}$.  Increase $u_{\beta}$ by 1, decrease $u_{\gamma}$ by 1, and jump down $(\bS_n, \preceq)$ to the resulting new sequence.  

We illustrate these jumps for $S = (2,2,2,3,3,3)$.  The appropriate $\beta$th and $\gamma$th positions are underlined in the following:
\vspace{-12pt}
\begin{center}
$(\underline{0},1,2,3,\underline{4},5) \succ (1,\underline{1},2,3,3,\underline{5}) \succ (\underline{1},2,2,3,3,\underline{4}) \succ (2,2,2,3,3,3)$.  
\end{center}
\vspace{-12pt}
Note that a sequence obtained by this algorithm need not be covered (in the total order $(\bS_n, \preceq)$) by the sequence from which it was obtained.  For example, 
\vspace{-12pt}
\begin{center}
$(1,1,2,3,3,5) \succ (0,1,2,4,4,4) \succ (1,2,2,3,3,4)$.
\end{center}
\vspace{-12pt}

In \cite{GriggsReid} it is shown that after $\frac{1}{2}d(Tr_n, S)$ such jumps, the sequence $S$ is reached, and $S$ is a score sequence.  If $S \neq R_n, Tr_n$, this algorithm jumps down $(\bS_n, \preceq)$ from $Tr_n$ to $S$, and the algorithm in Section \ref{proof} jumps down $(\bS_n, \preceq)$ from $S$ to $R_n$, so $S$ is the only common sequence.  Both algorithms can be employed to jump down $(\bS_n, \preceq)$ from $Tr_n$ to $R_n$, and the only resulting common sequence, besides $Tr_n$ and $R_n$, is the next-to-last sequence given by $\left(\dfrac{n-3}{2}, \dfrac{n-1}{2}, \dots, \dfrac{n-1}{2}, \dfrac{n+1}{2}\right)$, if $n$ is odd, or $\left(\dfrac{n-4}{2}, \dfrac{n-2}{2}, \dots, \dfrac{n-2}{2}, \dfrac{n}{2}, \dots, \dfrac{n}{2}, \dfrac{n+2}{2}\right)$, if $n$ is even.  So clearly, these two algorithms are quite distinct.  Moreover, if $S \in \bS_n$ is such that $\displaystyle\sum\limits_{i=1}^k s_i = \dbinom{k}{2}$ for some $k, 1 \le k < n$ and $s_{k+1} > k$, then the first algorithm in \cite{GriggsReid} eventually jumps to a sequence $S'$ in which the first $k$ terms are $s_1, s_2, \dots, s_k$ (in that order) strictly before reaching sequences in which the $(k+1)$st term is $s_{k+1}$, and all subsequent sequences obtained after $S'$ start with $s_1, s_2, \dots, s_k$.  In effect, $S$ is not a strong score sequence, so in any $n$-tournament $T$ with score sequence $S$, where $V(T) = \{v_1, v_2, \dots, v_n\}$ and the score of $v_i$ is $s_i, 1 \le i \le n$, the set $\{v_1, v_2, \dots, v_k\}$ induces a union of strong components of $T$, but $\{v_1, v_2, \dots, v_{k+1}\}$ does not.  This means that the scores of the strong components of $T$ are produced in order from the terminal component to the initial component, and once a sequence is reached that starts with $s_1, s_2, \dots, s_k$, all subsequent sequences start with $s_1, s_2, \dots, s_k$.  On the other hand, as described in Section \ref{proof}, for such an $S$, the very first jump in the algorithm in Section \ref{proof} amounts to reversing one path in $T$ from a vertex in the initial strong component of $T$ to a vertex in the terminal strong component resulting in a strong score sequence.

By an argument similar to that given in Section 4, the first algorithm in \cite{GriggsReid} also provides an $O(n^2)$ algorithm for constructing a tournament with score sequence $S$.  So, in practice, the efficiency of the algorithm in this paper compared to the first one described in \cite{GriggsReid} is determined by whether $S$ is closer to $Tr_n$ or $R_n$ via the metric $d(\cdot,\cdot)$.  That is, compare $\frac{1}{2}d(R_n,S)$, the number of jumps in the algorithm of this paper, with $\frac{1}{2}d(Tr_n,S)$, the number of jumps in the first proof in \cite{GriggsReid}.

For example, if $S= (1,1,1,4,4,4)$, then $S \in \bS_6$, $\frac{1}{2}d(Tr_6,S) = \frac{1}{2}(4) = 2$, while $d(R_6,S) = \frac{1}{2}(6) = 3$, but if $S = (1,2,3,3,3,3)$, then $S \in \bS_6$, $\frac{1}{2}d(Tr_6,S) = \frac{1}{2}(6) = 3$, while $\frac{1}{2}d(R_6,S) = \frac{1}{2}(2) = 1$.  It is possible for these two values to be equal.  For example, if $S = (2,2,2,3,4,4,4)$, then $\frac{1}{2}d(R_7,S) = \frac{1}{2}(6) = \frac{1}{2}d(Tr_7,S)$.

The second proof in \cite{GriggsReid} also describes jumps between sequences in $(\bS_n, \preceq)$, but these jumps are upward jumps towards the maximum sequence in $(\bS_n, \preceq)$, $Tr_n$.  To describe these jumps up, let $S \in \bS_n$, $S \neq Tr_n$.  Define $k$ to be the smallest index so that $s_k = s_{k+1}$, and define $m$ to be the number of occurrences of the value $s_k$ in $S$.  Reduce $s_k$ by 1, increase $s_{k+m-1}$ by 1, and jump up to the resulting new sequence.  Repeat until $Tr_n$ is obtained.  As shown in \cite{GriggsReid}, $S$ is a score sequence.  

We illustrate these jumps for $S=(1,1,3,3,3,4)$.  The appropriate $k$th and $k+m-1$th positions are underlined in the following:
\vspace{-12pt}
\begin{center}
$\begin{array}{l l l l l l}
&(\underline{1},\underline{1},3,3,3,4)& \prec& (0,2,\underline{3},3,\underline{3},4)& \prec& (0,\underline{2},\underline{2},3,4,4)\\
\prec &(0,1,\underline{3},\underline{3},4,4)& \prec& (0,1,2,\underline{4},4,\underline{4}) &\prec &(0,1,2,3,4,5).  
\end{array}$
\end{center}
\vspace{-12pt}

Again, a sequence obtained by this algorithm need not cover any sequence obtained previously by this algorithm in the total order $(\bS_6, \preceq)$; for example, $(0,2,3,3,3,4) \prec (1,2,2,2,4,4) \prec (0,2,2,3,4,4)$.  And, it is easy to see that most of the jumps upward from $R_n$ to $Tr_n$ using this algorithm are not the reverse of the jumps downward using either of the two algorithms discussed above.

The metric $d(\cdot,\cdot)$ used earlier is not useful for this algorithm in order to measure the number of jumps required to reach $Tr_n$ from $S \neq Tr_n$.  Indeed, 20 jumps are required for $S = (3,3,3,3,4,4,4,4)$, but $\frac{1}{2}d(Tr_8,S) = 6$.  This issue was not addressed in \cite{GriggsReid}.  For $S \in \bS_n$, let $c(S) = \dbinom{n}{3} - \displaystyle\sum\limits_{i=1}^n \dbinom{s_i}{2}$.  If $S', S'' \in \bS_n$ are such that $S''$ is obtained by a single jump up from $S'$ as described above, then 
\vspace{-12pt}
\begin{center}
$\begin{array}{l l l l}
&c(S') - c(S'') \\
=& - \left(\dbinom{s_k}{2} + \dbinom{s_k+m-1}{2}\right) + \left(\dbinom{s_k-1}{2} + \dbinom{s_{k+m-1}+1}{2}\right) \\
=& - \left(\dbinom{s_k}{2} - \dbinom{s_k-1}{2}\right) + \left(\dbinom{s_{k+m-1}+1}{2} - \dbinom{s_{k+m-1}}{2}\right) \\
=& - (s_k - 1) + s_{k+m-1} = 1,
\end{array}$
\end{center}
\vspace{-12pt}
since $s_k = s_{k+m-1}$.  That is, one up jump from $S'$ corresponds to a decrease of $c(S')$ by 1.  Thus, the number of up jumps in this algorithm applied to the initial sequence $S \in \bS_n$ is equal to $c(S)$.  Now, in any $n$-tournament with score sequence $S$, there are $\dbinom{n}{3}$ subtournaments of order 3 and $\dbinom{s_i}{2}$ transitive subtournaments of order 3 with transmitter at a vertex of score $s_i$.  A 3-tournament is either a cycle or transitive, so, $c(S)$ yields the number of cycles of length 3 in $T$.  Now, max$\{c(S) | S \in \bS_n\} = c(R_n) = \begin{cases}
  \frac{1}{24}(n^3 - n), & \text{ if $n$ is odd} \\
  \frac{1}{24}(n^3 - 4n), & \text{ if $n$ is even}\\
\end{cases}$ \cite{KendallBabington}.

Consequently, as $c(S) \le c(R_n)$ for all $S \in \bS_n$, the maximum number of up jumps that this algorithm will produce is obtained when the algorithm starts with $S = R_n$, and that number is given by the value of $c(R_n)$ above.  That is, it is a $O(n^3)$ algorithm.

\end{document}